\documentclass{amsart}
\usepackage{latexsym}
\usepackage{rotating}
\usepackage{amsfonts}
\usepackage{amssymb}
\usepackage{amsmath, dsfont}
\usepackage{graphics, color}
\usepackage{amsthm}
\usepackage{enumerate}
\input xy
\xyoption{all}

\DeclareMathOperator*{\hocolim}{hocolim\,}

\DeclareMathOperator*{\holim}{holim}

\DeclareMathOperator*{\Product}{\prod}

\newtheorem{theorem}{Theorem}[section]
\newtheorem{lemma}[theorem]{Lemma}
\newtheorem{corollary}[theorem]{Corollary}
\newtheorem{proposition}[theorem]{Proposition}

\newtheorem{example}[theorem]{Example}
\newtheorem{definition}[theorem]{Definition}
\newtheorem{remark}[theorem]{Remark}

\begin{document}

\newcommand{\mx}{\textrm{max}}
\newcommand{\bU}{\ul{\bf U}}
\newcommand{\UF}{\ul{\bf UF}}
\newcommand{\Gp}{{\mathrm{\bf Group}}}
\newcommand{\Cat}{{\mathrm{\bf Cat}}}
\newcommand{\Q}{\ul{\mathrm{\bf Quot}}}
\newcommand{\bbA}{\mathbb{A}}
\newcommand{\bbC}{\mathbb{C}}
\newcommand{\C}{\mathcal{C}}
\newcommand{\D}{\mathcal{D}}
\newcommand{\N}{\mathbb{N}}
\newcommand{\bbN}{\mathbb{N}}
\newcommand{\bbR}{\mathbb{R}}
\newcommand{\bbRP}{\mathbb{RP}}
\newcommand{\bbZ}{\mathbb{Z}}
\newcommand{\mcZ}{\mathcal{Z}}
\newcommand{\bbF}{\mathbb{F}}
\newcommand{\bbQ}{\mathbb{Q}}

\newcommand{\A}{\mathcal{A}}
\newcommand{\B}{\mathcal{B}}
\newcommand{\mcC}{\mathcal{C}}
\newcommand{\mcD}{\mathcal{D}}
\newcommand{\E}{\mathcal{E}}
\newcommand{\mF}{\mathcal{F}}

\newcommand{\bF}{\ul{\bf F}}

\newcommand{\F}{\mathcal{F}}
\newcommand{\G}{\mathcal{G}}
\newcommand{\mcG}{\mathcal{G}}
\newcommand{\mH}{\mathcal{H}}
\newcommand{\mcH}{\mathcal{H}} 
\newcommand{\I}{\mathcal{I}}
\newcommand{\mL}{\mathcal{L}}
\newcommand{\mcN}{\mathcal{N}}
\newcommand{\M}{\mathcal{M}}
\newcommand{\mO}{\mathcal{O}}
\newcommand{\mcP}{\mathcal{P}}
\newcommand{\mcR}{\mathcal{R}}
\newcommand{\T}{\mathcal{T}}
\newcommand{\U}{\mathcal{U}}
\newcommand{\V}{\mathcal{V}}
\newcommand{\Z}{\mathcal{Z}}

\newcommand{\bS}{\mathbf{S}}
\newcommand{\bd}{\mathbf{d}}
\newcommand{\nv}{\mathrm{N}_\cdot}
\newcommand{\nvk}{\mathrm{N}_k}
\newcommand{\bG}{\mathcal{G}_0} 
\newcommand{\sth}[1]{#1^{\mathrm{th}}}
\newcommand{\abs}[1]{\left| #1\right|}
\newcommand{\ord}[1]{\Delta \left( #1 \right)}
\newcommand{\leqs}{\leqslant}
\newcommand{\geqs}{\geqslant}
\newcommand{\heq}{\simeq}
\newcommand{\iso}{\simeq}
\newcommand{\maps}{\longrightarrow}
\newcommand{\lmaps}{\longleftarrow}
\newcommand{\injects}{\hookrightarrow}
\newcommand{\homeo}{\cong}
\newcommand{\surjects}{\twoheadrightarrow}
\newcommand{\isom}{\cong}
\newcommand{\cross}{\times}
\newcommand{\normal}{\vartriangleleft}
\newcommand{\wt}[1]{\widetilde{#1}} 
\newcommand{\fc}{\mathcal{A}_{\mathrm{flat}}} 
\newcommand{\flc}{\mathcal{A}_{\mathrm{fl}}}

\newcommand{\Rdef}{R^{\mathrm{def}}}
\newcommand{\Sym}{\textrm{Sym}}
\newcommand{\vect}[1]{\stackrel{\rightharpoonup}{\mathbf #1}}
\newcommand{\SR}{\mathcal{SR}}
\newcommand{\SRe}{\mathcal{SR}^{\mathrm{even}}}
\newcommand{\Rep}{\mathrm{Rep}}
\newcommand{\SRep}{\mathrm{SRep}}
\newcommand{\Hom}{\mathrm{Hom}}
\newcommand{\HHom}{\mathcal{H}\mathrm{om}}
\newcommand{\Lie}{\mathrm{Lie}}
\newcommand{\K}{K^{\mathrm{def}}}
\newcommand{\mK}{\mathcal{K}_{\mathrm{def}}}
\newcommand{\SK}{SK_{\mathrm{def}}}
\newcommand{\dom}{\mathrm{dom}}
\newcommand{\codom}{\mathrm{codom}}
\newcommand{\Ob}{\mathrm{Ob}}
\newcommand{\Mor}{\mathrm{Mor}}
\newcommand{\+}[1]{\underline{#1}_+}
\newcommand{\ul}[1]{\underline{#1}}
\newcommand{\hofib}{\mathrm{hofib}}
\newcommand{\Stab}{\mathrm{Stab}}
\newcommand{\wtStab}{\wt{\mathrm{Stab}}}
\newcommand{\Css}{\mathcal{C}_{ss}}
\newcommand{\Map}{\mathrm{Map}}
\newcommand{\bMap}{\mathrm{Map_*}}
\newcommand{\bdMap}{\mathrm{Map_*^\delta}}
\newcommand{\flatc}{\mathcal{A}_{\mathrm{flat}}}
\newcommand{\p}{\vect{p}}
\newcommand{\avg}{\mathrm{avg}}
\newcommand{\smsh}[1]{\ensuremath{\mathop{\wedge}_{#1}}}
\newcommand{\ol}[1]{\overline{#1}}
\newcommand{\Vect}{\mathrm{Vect}}
\newcommand{\bv}{\bigvee}
\newcommand{\Gr}{\mathrm{Gr}}
\newcommand{\Mf}{\mathcal{M}_{\textrm{flat}}}
\newcommand{\ku}{\mathbf{ku}}
\newcommand{\Susp}{\Sigma}
\newcommand{\Id}{\textrm{Id}}
\newcommand{\id}{\textrm{Id}}
\newcommand{\xmaps}{\xrightarrow}
\newcommand{\srm}[1]{\stackrel{#1}{\maps}}
\newcommand{\srt}[1]{\stackrel{#1}{\to}}
\newcommand{\sm}{\wedge}
\newcommand{\conv}{\Rightarrow}
\newcommand{\Tor}{\textrm{Tor}}
\newcommand{\goesto}{\mapsto}
\newcommand{\nd}{\noindent}
\newcommand{\Ind}{\mathrm{Ind}}
\newcommand{\R}{\mathrm{R}}
\newcommand{\bR}{\overline{\mathrm{R}}}
\newcommand{\bRf}{\overline{\mathrm{R}}^{\mathrm{free}}}
\newcommand{\ra}{\rangle}
\newcommand{\la}{\langle}
\newcommand{\Sum}{\mathrm{Sum}}
\newcommand{\Res}{\mathrm{Res}}
\newcommand{\Proj}{\mathrm{Proj}}
\newcommand{\GL}{\mathrm{GL}}
\newcommand{\PU}{\mathrm{PU}}
\newcommand{\Irr}{\mathrm{Irr}}
\newcommand{\Irrp}{\mathrm{Irr}^+}
\newcommand{\rHom}{\Hom_A (\Gamma, U(n))}
\newcommand{\qcd}{\mathrm{qcd}}
\newcommand{\rk}{\mathrm{rk}}
\newcommand{\Irrf}{\Irr_n (H)^{\mathrm{free}}}
\newcommand{\brho}{\overline{\rho}}
\newcommand{\Sp}{\mathrm{Sp}}
\newcommand{\Span}{\mathrm{Span}}
\newcommand{\Orb}{\mathrm{Orb}}
\newcommand{\Fin}{{\mathcal{F}in}}
\newcommand{\MorF}{\mathrm{Mor}_{\Orb_\F(\Gamma)}}
\newcommand{\Sets}{\mathrm{\bf Set}}
\newcommand{\Top}{\mathrm{\bf Top}}
\newcommand{\op}{\mathrm{op}}
\newcommand{\down}{\downarrow}

\newcommand{\e}{\emph}

\def\co{\colon\thinspace}

\title[Generalized homotopy fixed points]{A note on orbit categories, classifying spaces, and generalized homotopy fixed points}
\author[D\,A Ramras]{Daniel A. Ramras}
\address{Department of Mathematical Sciences, Indiana University-Purdue University 
Indianapolis, 402 N. Blackford, LD 270, Indianapolis, IN 46202, USA}
\email{dramras@iupui.edu}
\urladdr{http://www.math.iupui.edu/~dramras/}


\thanks{This work was partially supported by a grant from the Simons Foundation (\#279007).}

 \begin{abstract}
We give a new description of Rosenthal's generalized homotopy fixed point spaces as homotopy limits over the orbit category.  This is achieved 
using a simple categorical model for classifying spaces with respect to families of subgroups.  
\end{abstract}

\maketitle{}


\section{Introduction}

Let $\Gamma$ be a discrete group.  A \e{family} of subgroups of $\Gamma$ is a non-empty set $\F$ of subgroups of $\Gamma$ such that 1) if $F\in\F$ and $\gamma\in \Gamma$, then $\gamma F \gamma^{-1} \in \F$; and 2) if $F\in \F$ and $G\leqs F$, then $G\in \F$.
Given a family $\F$
of subgroups
 of $\Gamma$,  a \e{classifying space} for $\F$ is a $\Gamma$--CW complex $\E_\F \Gamma$ whose fixed point sets $(\E_\F \Gamma)^H$ are contractible when $H\in \F$ and empty when $H\notin \F$.  
 If $E$ and $E'$ are both classifying spaces for $\F$, then there exists a $G$--homotopy equivalence $E\xmaps{\heq} E'$~\cite[Section 1]{Luck-survey}.

In this note, we highlight an extremely simple model for $\E_\F \Gamma$
and use it to give a new construction of Rosenthal's generalized homotopy fixed point sets~\cite{Rosenthal}.  This model comes from a category $E_\F \Gamma$, on which $\Gamma$ acts by functors, whose geometric realization $|E_\F \Gamma|$ is a functorial model for $\E_\F \Gamma$ (see Proposition~\ref{classifying-space} and Proposition~\ref{functorial}).   

While this model is not new, it seems not to be as widely known as it deserves to be.  The construction described here is commonly used for \e{collections} of subgroups, which are closed under conjugation but not necessarily under passage to subgroups.  In that context, $\E_\F \Gamma$ is used to study homology decompositions for finite groups; see in particular~\cite{Dwyer-homology-decompositions, Grodal}.  The nerve of $E_\F \Gamma$ is precisely the bar construction model for $\E_\F \Gamma$ presented by Elmendorf~\cite[pp. 277-278]{Elmendorf}, where methods of equivariant topology are used to show this simplicial set is a classifying space for $\F$; in the Appendix to the recent preprint~\cite{MNN}, Mathew, Naumann, and Noel give another proof of this fact using properties of homotopy colimits.  In Section~\ref{orbit-cat}, we give a short categorical proof that $|E_\F \Gamma|$ is a classifying space for $\F$, similar to Dwyer~\cite[Corollary 2.15]{Dwyer-homology-decompositions}.  
 
We give two applications.  
As observed by Grodal~\cite[Corollary 2.10]{Grodal}, 
the quotient space $|E_\F \Gamma|/\Gamma$ is the geometric realization of the orbit category $\Orb_\F (\Gamma)$ (namely, the full subcategory of the orbit category of $\Gamma$ on the objects $\Gamma/F$ with $F\in \F$; see Section~\ref{orbit-cat}). 
It then follows from work of Leary and Nucinkis~\cite{Leary-Nucinkis} that every connected CW complex is homotopy equivalent to $\Orb_\Fin (\Gamma)$ for some $\Gamma$, where $\Fin = \Fin (\Gamma)$ is the family of finite subgroups.
In Section~\ref{fixed-sets} we use the categorical viewpoint on classifying spaces for families to show that generalized homotopy fixed point sets can be described as homotopy limits over the orbit category. Rosenthal's homotopy invariance result for these spaces is then a formal consequence.  Generalized homotopy fixed point sets play an important role in the study of assembly maps in algebraic $K$-- and $L$--theory~\cite{Bartels-domain, Bartels-Rosenthal, Kasprowski}, and also appear  in~\cite{MNN} (in the form of the homotopy limit construction presented in Section~\ref{fixed-sets}), where they are used to study induction and restriction maps for equivariant spectra.

\vspace{.15in}

\nd {\bf Acknowledgements:} The author thanks Jesper Grodal for pointing out earlier appearances of the category $E_\F \Gamma$ in the literature, and Mark Ullmann and David Rosenthal for helpful comments. Additionally, the author thanks the referee and editor for helping to improve the exposition.

%
%


\section{The orbit category and a categorical model for $\E_\F \Gamma$} $\label{orbit-cat}$

For a fixed family of subgroups $\F$ of a group $\Gamma$, the orbit category $\Orb_\F (\Gamma)$ is the category whose objects are the ``homogeneous" left $\Gamma$--sets $\Gamma/F$, for $F\in \F$ (so the objects are in explicit bijection with $\F$ itself) and where 
$$\MorF (\Gamma/F, \Gamma/G) = \Map^\Gamma (\Gamma/F, \Gamma/G),$$
the set of $\Gamma$--equivariant maps between these left $\Gamma$--sets.

\begin{lemma} $\label{morphisms}$
For each $F, G\in \F$, there is a bijection
$$\xi \co \MorF (\Gamma/F, \Gamma/G) \srm{\isom} (\Gamma/G)^F
= \{\gamma G  \,:\, \gamma^{-1} F\gamma \subset G\}.$$
defined by $\xi (\phi)  =\phi(1F)$.  Here $(\Gamma/G)^F$ denotes the $F$--fixed points of $\Gamma/G$ under the left multiplication action of $F\leqs \Gamma$.
Hence we may write morphisms $\Gamma/F\to \Gamma/G$  in $\Orb_\F (\Gamma)$ as equivalence classes $[\gamma] = \gamma G$ of elements in $\Gamma$. 
\end{lemma}

\begin{example} Let $\mathds{1}$ denote the trivial family containing only the trivial subgroup of $\Gamma$.  Then $\Orb_{\mathds{1}} (\Gamma)$ 
is isomorphic to the usual one-object category modeling the classifying space $B\Gamma = K(\Gamma, 1)$.
\end{example}

\begin{definition}\label{UF} The functor $\U_\F: \Orb_\F (\Gamma) \to \Sets$ sends each object $\Gamma/F$ to the set $\Gamma/F$, and sends each morphism $\Gamma/F \to \Gamma/G$ to its underlying equivariant map.

We define $E_\F \Gamma$ to be the Grothendieck wreath product category $\Orb_\F (\Gamma) \wr \U_\F$.
\end{definition}
 
We briefly recall the definition of the wreath product $\C \wr F$ of a category $\C$ with a functor $F: \C\to \Sets$.  This category has objects all pairs $(C, x)$ with $C\in \Ob (C)$ and $x\in F(C)$, and morphisms $(C,x) \to (D,y)$ are pairs $(\phi, x)$ where $\phi\in \Mor_\C (C, D)$ and $F(\phi)(x) = y$.  (We will often abbreviate $(\phi, x)$ to $\phi$.)  Composition is given by $(\psi, y) \circ (\phi, x) = (\psi\circ \phi, x)$. 

Objects of $E_\F \Gamma = \Orb_\F (\Gamma) \wr \U_\F$ are pairs $(\Gamma/F, \gamma F)$ with $\gamma\in \Gamma$ and $F\in \F$.  Note that there is a natural ``forgetful" functor $E_\F \Gamma\to \Orb_\F (\Gamma)$ sending $(\Gamma/F, \gamma F)$ to $\Gamma/F$ and sending $\phi: \Gamma/F \to \Gamma/G$ to itself.

Since a subset of a group cannot be a left coset of more than one subgroup, we can describe $E_\F \Gamma$ more explicitly. 

\begin{lemma}$\label{E_F}$ The category $E_\F \Gamma$ is isomorphic to the category whose objects are left cosets $\gamma F$ satisfying $\gamma\in \Gamma$ and $F\in \F$, and whose  morphisms $\gamma F \to \nu G$ are $\Gamma$--equivariant maps $\phi: \Gamma/F \to \Gamma/G$ such that $\phi(\gamma F) = \nu G$, with the usual composition.
\end{lemma}

We will sometimes treat this isomorphism as an identification in what follows.

\begin{example} 
The category $E_{\mathds{1}} \Gamma$ is the usual categorical model for the universal principal $\Gamma$--bundle $E\Gamma$ $($see~\cite{Segal-class-ss}, for instance$)$, and the forgetful functor $E_{\mathds{1}} \Gamma \to \Orb_{\mathds{1}} (\Gamma)$ realizes to the projection $E\Gamma\to B\Gamma$.
\end{example}

\begin{remark} $\label{morphisms-rmk}$ Since $\Gamma$ acts transitively on each set $\Gamma/F$, there is at most one morphism in $E_\F \Gamma$ between any two objects.
Hence all diagrams in $E_\F \Gamma$ commute.
\end{remark}

We now study the action of $\Gamma$ on $E_\F \Gamma$.  Our discussion will be framed in terms of homotopy colimits.
Given a functor $\bF \co \C\to \Cat$ (the category of small categories) one may construct a wreath product category $\C\wr \bF$ specializing to the construction considered above when $\bF$ takes values in $\Sets$ (that is, discrete categories).
Thomason~\cite{Thomason-thesis} shows that in general, there is a natural map
of simplicial sets
$$\hocolim_\C \nv \bF\maps \nv (\C\wr \bF)$$
which induces a homotopy equivalence on geometric realizations.
In the case where $\bF$ takes values in $\Sets$, 
$\hocolim_\C \bF$ is the simplicial set whose $n$--simplices consist of an $n$--simplex $C_0 \to \cdots \to C_n$ in $\nv \C$ together with an element $x\in \bF(C_0)$, and one finds that there is in fact a natural \e{isomorphism}
\begin{equation}\label{nv} \hocolim_\C \bF\isom  \nv (\C\wr \bF) .
\end{equation}

Let $\Gamma$--$\Sets$ denote the category of (left) $\Gamma$--sets and equivariant maps.  We then have two functors
$\Gamma$--$\Sets \to \Sets$, the forgetful functor $\bU$ and the quotient set functor $\Q$.
Note that the functor $\U_\F$ in Definition~\ref{UF} is simply the composite functor
$$\Orb_\F (\Gamma) \injects \Gamma\textrm{--}\Sets \srm{\bU} \Sets,$$
where the first functor is the natural inclusion.
Now consider a functor 
$$\bF \co \C\to \Gamma\textrm{--}\Sets.$$
Letting $\UF := \bU \circ \bF$ and $\bF/\Gamma := \Q \circ \bF$, 
Equation (\ref{nv})  gives two natural isomorphisms:
\begin{equation}\label{nv2}\nv (\C\wr \UF) \isom \hocolim_\C  \UF \,\,\,\textrm{ and } \,\,\,
\nv (\C\wr (\bF/\Gamma)) \isom \hocolim_\C  \bF/\Gamma.
\end{equation}

The wreath product category
$\C \wr \UF$
inherits an action of $\Gamma$ by functors 
\begin{equation}\label{act}\ul{\gamma}\co \C \wr \UF\maps \C \wr \UF\end{equation}
for each $\gamma \in \Gamma$: on objects, we set $\ul{\gamma} (C, x) = (C, \gamma\cdot x)$, and on morphisms we set $\ul{\gamma} ( \phi) = \phi$.  The nerve $\nv (\C \wr \bF)$ is then a $\Gamma$--simplicial set, and $|\C \wr \bF|$ is a $\Gamma$--CW complex.

In general, for a functor $\bF \co \C \to \Gamma$--$\Sets$, it follows immediately from the definitions that there is a natural isomorphism of simplicial sets 
\begin{equation*}\label{h1}(\hocolim_\C \UF)/\Gamma \srm{\isom} \hocolim_\C (\bF/\Gamma)
\end{equation*} 
induced by the canonical natural transformation $\UF\to \bF/\Gamma$.  On the left, $\Gamma$ acts simplicially on $\hocolim_\C \UF$ and the quotient can be formed level-wise; on the other hand, this quotient object is the colimit, in the category of simplicial sets, of the $\Gamma$--action.  Geometric realization commutes with colimits, giving a homeomorphism
\begin{equation}\label{h2}|\hocolim_\C \UF|/\Gamma \srm{\isom} |\hocolim_\C (\bF/\Gamma)|.
\end{equation} 
Since the actions of $\Gamma$ on $\nv (\C\wr \UF)$ and on $\hocolim_\C \UF$ are compatible under the isomorphism (\ref{nv}), combining 
(\ref{nv2}) and (\ref{h2}) gives a chain of natural homeomorphisms
\begin{equation}\label{quotients}
|\nv (\C\wr \UF)|/\Gamma \homeo |\hocolim_\C \UF|/\Gamma \homeo
|\hocolim_\C (\bF/\Gamma)| \homeo |\nv (\C\wr (\bF/\Gamma))|.
\end{equation} 

\begin{remark} The category $\C\wr (\bF/\Gamma)$ is the colimit, in $\Cat$, of the $\Gamma$--action on $\C \wr \bF$ $($as can be checked using the universal property of colimits$)$.  However, the nerve functor does not always commute with colimits $($even for group actions~\cite{Babson-Kozlov-gp-actions}$)$ so  $(\ref{quotients})$ does not follow immediately. \end{remark}

We need to understand the fixed points for the action of $\Gamma$ on $E_\F \Gamma$.

\begin{lemma}$\label{fixed}$
Let $G\leqs \Gamma$ be a subgroup.  The subcategory $(E_\F \Gamma)^G$ of $E_\F \Gamma$, consisting of objects and morphisms fixed by $G$, is the full subcategory on the objects 
 $(\Gamma/F, \gamma F)$, with $F\in \F$ and $\gamma \in \Gamma$, such that $\gamma^{-1} G \gamma \leqs F$.
\end{lemma}

\begin{proof} Under the isomorphism in Lemma~\ref{E_F}, the action of $G$ on the objects of $E_\F \Gamma$ translates to the left multiplication action of $\Gamma$ on cosets, and 
$\gamma F \in (\Gamma/F)^G$ if and only if $\gamma^{-1} G \gamma \leqs F$. Hence the objects of  $(E_\F \Gamma)^G$ are as claimed. 
It is immediate from the definition of the action that $(E_\F \Gamma)^G$ is a full subcategory.
\end{proof}

We come now to the main result of this section.

\begin{proposition}$\label{classifying-space}$ The geometric realization $|E_\F \Gamma|$ of the category $E_\F \Gamma$  is a 
classifying space for the family $\F$, and the forgetful functor $E_\F \Gamma \to \Orb_\F (\Gamma)$ induces a homeomorphism
$$|E_\F \Gamma|/\Gamma \isom |\Orb_\F (\Gamma)|.$$  
\end{proposition}
\begin{proof} 
If a group $H$ acts on a small category $\C$, then the inclusion $i:|\C^H| \injects |\C|^H$ is surjective, hence a homeomorphism.  Indeed, each $x\in |\C|^H$ lies in a unique non-degenerate simplex $\sigma$ in $\nv \C$.  Since $H$ fixes $x$, it fixes $\sigma$, so $\sigma\in \nv (\C^H)$.  

It remains to show that $(E_\F \Gamma)^G$ is empty for $G\notin \F$ and contractible for $G\in \F$.
If $(E_\F \Gamma)^G$ is non-empty, then some object $(\Gamma/F, \gamma F)$ in $E_\F \Gamma$ is fixed by $G$.  By Lemma~\ref{fixed}, we know that $\gamma^{-1} G \gamma \leqs F$, and
since $\F$ is a family of subgroups, it follows that $G \in \F$.  So $(E_\F \Gamma)^G = \emptyset$ when $G\notin \F$.
Next, one checks (using Lemmas~\ref{fixed} and~\ref{morphisms} and Remark~\ref{morphisms-rmk}) that for each $F\in \F$, the object $(\Gamma/F, 1F)$ is initial in $(E_\F \Gamma)^{F}$.  Hence $|(E_\F \Gamma)^{F}|$ is contractible~\cite[Section 1]{Quillen}.

Finally, since each of the $\Gamma$--sets $\Gamma/F$ is transitive, 
the functor $\U_\F/\Gamma$ sends each object of $\Orb_\F (\Gamma)$ to a one-element set. 
Hence $\Orb_\F (\Gamma) \wr (\U_\F/\Gamma) \isom \Orb_\F (\Gamma)$, and
(\ref{quotients}) completes the proof.
\end{proof}

\begin{proposition}$\label{functorial}$ The construction in Proposition~\ref{classifying-space} is functorial in the following sense.  
Group homomorphisms $\Gamma' \xmaps{h} \Gamma$ induce functors  
$$E_{h^{-1} (\F)} (\Gamma') \srm{h_\F} E_\F \Gamma,$$
where $h^{-1} (\F)$ is the family $\{G\leqs \Gamma' \,:\, h(G) \in \F\}$.
These functors are equivariant in the sense that for each $\gamma'\in \Gamma'$, we have
$$\ul{h(\gamma')} \circ h_\F = h_\F \circ \ul{\gamma'},$$
where $\ul{\gamma'}$ and $\ul{h(\gamma')}$ are the functors defined in $(\ref{act})$.

Moreover, for all homomorphisms $\Gamma''\xmaps{k} \Gamma' \xmaps{h} \Gamma$, we have 
$$(h\circ k)_\F = h_\F \circ k_{h^{-1} \F}.$$
\end{proposition}

The proof relies on the following elementary lemmas, the first of which is noted in~\cite[\S 1]{Thomason-thesis}.   

\begin{lemma}\label{func1} A natural transformation $\eta\co F\implies G$ of functors $F, G\co \C\to \Sets$ induces a functor $$\ul{\eta} \co \C\wr F \to \C \wr G,$$ defined on objects by $\ul{\eta}(C, x) = (C, \eta_C (x))$ and on morphisms by $\ul{\eta} (\phi) = \phi$.

Given another natural transformation $\eta' \co G\implies H$, we have $ \ul{\eta' \circ \eta}=\ul{\eta'} \circ \ul{\eta}$.
\end{lemma}

\begin{lemma}\label{func2} Functors $\B\srt{\beta} \C \srt{F} \Sets$ induce a functor 
$$\ol{\beta} \co \B \wr (F\circ \beta) \maps \C \wr F$$
defined on objects by $\ol{\beta} (B, x) = (\beta (B), x)$ and on morphisms by $\ol{\beta} (\phi) = \beta(\phi)$.

Given $\A \srt{\alpha} \B\srt{\beta} \C \srt{F} \Sets$, we have $\ol{\beta \circ \alpha} = \ol{\beta} \circ \ol{\alpha}$.
\end{lemma}

\begin{lemma}\label{func3}
The constructions in Lemma~\ref{func1} and Lemma~\ref{func2} commute, in the following sense: given functors
$$\xymatrix{ \A \ar[r]^-\alpha & \B \ar@<0.5ex>[r]^-F   \ar@<-0.5ex>[r]_-G
		&	\Sets
}$$
and a natural transformation $\eta \co F\implies G$, the diagram
$$\xymatrix{ \A \wr (F\circ \alpha) \ar[r]^-{\ol{\alpha}} \ar[d]^-{\ul{\eta \circ \alpha}}
					& \B\wr F \ar[d]^-{\ul{\eta}} \\
		 \A\wr (G \circ \alpha) \ar[r]^-{\ol{\alpha}} & \B \wr G
		 }$$
commutes $($strictly$)$, where $\eta \circ \alpha\co F\circ \alpha \implies G\circ \alpha$ is the natural transformation $(\eta\circ \alpha)_A = \eta_{\alpha (A)}$.
\end{lemma}

\begin{proof}[Proof of Proposition~\ref{functorial}]
To define $h_\F$, first note that $h$ induces a functor
$$\Orb_{h^{-1} (\F)} (\Gamma') \srm{h_*} \Orb_\F (\Gamma),$$
defined by $h_* (\Gamma'/G) = \Gamma/h(G)$ and $h_* (\gamma'G) = h(\gamma')h(G)$ $($using the notation from Lemma~\ref{morphisms}$)$.  The diagram
$$\xymatrix{ \Orb_{h^{-1} (\F)} (\Gamma') \ar[dr]_-{\U_{h^{-1} (\F)}} \ar[rr]^{h_*} & & \Orb_\F (\Gamma) \ar[dl]^-{\U_\F}\\
& \Sets
}$$
commutes up to the natural transformation 
$$\eta_{h, \F} \co \U_{ h^{-1} (\F)} \implies \U_\F \circ h_*$$
whose value on an object $\Gamma'/G$ in $\Orb_{h^{-1} (\F)} (\Gamma')$ is the function $\Gamma'/G\to \Gamma/h(G)$ given by 
$\gamma' G\goesto h(\gamma') h(G)$.  The desired functor $h_\F$ is the composite
$$\Orb_{h^{-1} (\F)} (\Gamma') \wr \U_{ h^{-1} (\F)} \xmaps{\ul{\eta_{h, \F}}} \Orb_{h^{-1} (\F)} (\Gamma') \wr \left(\U_{ h^{-1} (\F)} \circ h_*\right) \xmaps{\ol{h_*}} \Orb_\F (\Gamma) \wr \U_\F,$$
where $\ul{\eta_{h, \F}}$ is the functor from Lemma~\ref{func1} and $\ol{h_*}$ is the functor from Lemma~\ref{func2}. 

The category $ \Orb_{h^{-1} (\F)} (\Gamma') \wr \left(\U_\F\circ h_*\right) $ admits an action of $\Gamma$, defined analogously to the action of $\Gamma'$ on $\Orb_{h^{-1} (\F)} (\Gamma') \wr \U_{ h^{-1} (\F)}$. This $\Gamma$--action extends, via the homomorphism $h\co \Gamma'\to \Gamma$, to an action of $\Gamma'$ on $\Orb_{h^{-1} (\F)} (\Gamma') \wr \left(\U_{ h^{-1} (\F)}\circ h_*\right)$, and one checks that $\ul{\eta_{h, \F}}$ is $\Gamma'$--equivariant, while $\ol{h_*}$ is $\Gamma$--equivariant. The stated equivariance property for $h_\F$ now follows.

Applying Lemma~\ref{func3} followed by the naturality statements in Lemmas~\ref{func1} and \ref{func2} gives
\begin{eqnarray*}
h_\F\circ k_{h^{-1}(\F)} &= &\ol{h_*} \circ \ul{\eta_{h, \F}}  \circ \ol{k_*} \circ \ul{\eta_{k, h^{-1} (\F)}} \\
&=& \ol{h_*} \circ  \ol{k_*} \circ  \ul{\eta_{h, \F} \circ k_*}\circ  \ul{\eta_{k, h^{-1} (\F)}}\\
&= &\ol{h_* \circ k_*} \circ \ul{ (\eta_{h, \F} \circ k_*)\circ  \eta_{k, h^{-1} (\F)}}.
\end{eqnarray*}
Since $h_* \circ k_* = (h\circ k)_*$ and  
$(\eta_{h, \F} \circ k_*)\circ  \eta_{k, h^{-1} (\F)} = \eta_{h\circ k, \F}$, we find that 
$$h_\F\circ k_{h^{-1}(\F)} = (h\circ k)_\F,$$ 
as desired.
\end{proof}

\begin{remark} By Equation $(\ref{nv})$, the model for classifying spaces considered here satisfies
$$E_\F \Gamma \isom \hocolim_{\Orb_\F (\Gamma)} \U_\F.$$
Davis and L\"{u}ck~\cite[Lemma 7.6(2)]{Davis-Luck} give another proof that $\E_\F \Gamma$ can be described as the homotopy colimit of $\U_\F$ $($which they denote by $\nabla$$)$ over the orbit category.
\end{remark}

We note an interesting corollary of Proposition~\ref{classifying-space}.  Recall that $\Fin$ is the family of finite subgroups, so $E_\Fin (\Gamma) = \ul{E} \Gamma$ is the \e{classifying space for proper actions}.

\begin{corollary}  Let $X$ be a connected CW complex.  Then there exists a discrete group $\Gamma$ such that $|\Orb_{\Fin} (\Gamma)|$ is homotopy equivalent to $X$.
\end{corollary}

\begin{proof} The main result of \cite{Leary-Nucinkis} states that for every connected CW complex $X$, there exists a group $\Gamma_X$ and a classifying space $E_X$ for the family $\Fin = \Fin(\Gamma_X)$
such that $E_X/\Gamma_X$ is homotopy equivalent to $X$.
By Proposition~\ref{classifying-space}, $|E_\Fin (\Gamma_X)|$ is also a classifying space for $\Fin$, so there exists a $\Gamma_X$--homotopy equivalence 
 $$|E_\Fin (\Gamma_X)| \srm{\heq} E_X.$$
It follows that the induced mapping 
$$|E_\Fin (\Gamma_X)|/\Gamma_X \homeo |\Orb_\Fin (\Gamma_X)| \maps E_X/\Gamma_X\heq X$$
is also a homotopy equivalence.
\end{proof}


\section{Generalized homotopy fixed points via the orbit category}$\label{fixed-sets}$

Generalized homotopy fixed point sets were  introduced by Rosenthal \cite{Rosenthal}.
We show that these spaces can be described as homotopy limits over the orbit category.
The results in this section are well-known when $\F$ is the trivial family $\mathds{1}$.

\begin{definition} $\label{gfs}$ Given a $($left$)$ $\Gamma$--space $X$, the \emph{generalized homotopy fixed point set} with respect to the family $\F$ is the equivariant mapping space 
$$X^{h_\F \Gamma} := \Map^\Gamma (\E_\F \Gamma, X).$$  
\end{definition}

Since all models for $\E_\F \Gamma$ are related by $\Gamma$--homotopy equivalences, 
$X^{h_\F \Gamma}$ is well-defined up to homotopy.  
We will work in the category of compactly generated spaces, 
so we will always replace the natural topology on a space by the associated compactly generated topology.  In particular, $X^{h_\F \Gamma}$ has the compactly generated topology associated to the subspace topology inherited from the full mapping space.  
We   assume familiarity with the basic results on compactly generated spaces described in~\cite{Steenrod-convenient}.

We review Goodwillie's viewpoint on homotopy limits from~\cite{Goodwillie-calculus-II}.  
For a small category $\C$ and an object $C\in \Ob (\C)$, the over-category $\C \down C$ has as its objects the set of arrows $D\xmaps{\phi} C$ in $\C$, and morphisms 
in $\C\down C$ from $D\xmaps{\phi} C$ to $D'\xmaps{\phi'} C$
are  arrows $D\xmaps{\psi} D'$ in $\C$ such that $\phi' \circ \psi = \phi$.
Given a functor $\ul{X}\co \C \to \Top$, the homotopy limit
$$\holim_\C \ul{X}$$
can be described as the space of all \e{natural} collections of maps $f_C\co |\C \down C| \to\ul{X}(C)$,
topologized using the (compactly generated topology associated to the) subspace topology from the product space 
$$\Product_{C\in \Ob (\C)} \Map \left(|\C\down C|, \ul{X} (C)\right).$$
A collection of maps $|\C \down C| \to \ul{X}(C)$ is natural if 
for each $\alpha: C\to C'$, the diagram
$$\xymatrix{ |\C \down C| \ar[d]^{\ul{\alpha}} \ar[r]^{f_C} 
					& \ul{X}(C)  \ar[d]^{\ul{X} (\alpha)}  \\
		    |\C\down C'| \ar[r]^{f_{C'}} & \ul{X} (C')
		}
$$
commutes, where $\ul{\alpha}: \C \down C \to \C \down C'$ is induced by composition with $\alpha$.

Given a $\Gamma$--space $X$, the fixed point sets form a diagram (that is, a functor)
$$X^- : \Orb_\F (\Gamma)^\op \maps \Top.$$
On objects, this diagram is defined by $X^- (\Gamma/F) = X^F \isom \Map^\Gamma (\Gamma/F, X)$.  The mapping space is contravariantly functorial in $\Gamma/F$, which defines $X^-$ on morphisms.  Explicitly, if $\phi = [\gamma]: \Gamma/F \to \Gamma/G$ is an equivariant map, then 
$X^-(\phi) (x) = \gamma\cdot x$.

In what follows, it will be useful to note the isomorphism
$$(\C^\op \down C)^\op \isom C\down \C.$$

\begin{theorem}$\label{holim}$ If $X$ is a $\Gamma$--space, then there is a homeomorphism 
$$\holim_{\Orb_\F (\Gamma)^\op} X^- \homeo \Map^\Gamma (|E_\F \Gamma|, X).$$
\end{theorem}
\begin{proof}  
Given a $\Gamma$--equivariant map $\Phi: |E_\F \Gamma| \to X$, we need to define a natural collection of maps 
$$\Phi_F : |\Orb_\F (\Gamma)^\op \down \Gamma/F| \maps X^F$$
for $F\in \F$.  These maps are defined via functors 
$$j_F : \left(\Orb_\F (\Gamma)^\op \down \Gamma/F\right)^\op \isom \Gamma/F \down \Orb_\F (\Gamma) \maps (E_\F \Gamma)^F.$$
An object in $\Gamma/F \down \Orb_\F (\Gamma)$ is an equivariant map $\phi: \Gamma/F \to \Gamma/G$, and we define 
$$j_F (\phi) = (\Gamma/G, \phi(1F)),$$
which is an object of $(E_\F \Gamma)^F$ since $\phi(1F) \in (\Gamma/G)^F$.
Morphisms 
$$\left(\phi: \Gamma/F \to \Gamma/G\right) \maps \left(\psi: \Gamma/F \to \Gamma/H\right)$$
in $\Gamma/F \down \Orb_\F (\Gamma)$
are equivariant maps $\tau: \Gamma/G\to \Gamma/H$ such that $\tau\circ \phi = \psi$.  
Viewing $\tau$ as a morphism
$(\Gamma/G,  \phi(1F)) \to (\Gamma/H, \psi (1F))$ 
in $E_\F \Gamma$, we define $j_F (\tau) = \tau$.
All diagrams in $E_\F \Gamma$ commute (Remark~\ref{morphisms-rmk}), so
$j_F$ is a functor.

We claim that the family of maps 
$$\Phi_F := \Phi^F \circ |j_F| :  |\Orb_\F (\Gamma)^\op \down \Gamma/F| \isom |\left( \Orb_\F (\Gamma)^\op \down \Gamma/F\right)^\op| \maps X^F$$
is natural, where $\Phi^F$ is the restriction of $\Phi$ to $F$--fixed point sets. 
It suffices to check that for every morphism
$\alpha \co \Gamma/F' \to \Gamma/F$ in $\Orb_\F \Gamma$, the diagram
$$\xymatrix{ |\Orb_\F (\Gamma)^\op \down \Gamma/F| \ar[r]^-{|j_F|} \ar[d]^{|\ul{\alpha}|}
					& | (E_\F \Gamma)^F | \ar[r]^-{\Phi^F} \ar[d]^-{|\ul{\gamma}|} 
					& X^F \ar[d]^{X^-(\alpha)} \\
		 |\Orb_\F (\Gamma)^\op \down \Gamma/F'| \ar[r]^-{|j_{F'}|} 										& | (E_\F \Gamma)^{F'} | \ar[r]^-{\Phi^{F'}} & X^{F'}
		}
$$ 
commutes, where $\gamma$ is chosen so that $\alpha (1F') = \gamma F$
and $\ul{\gamma}$ is the restriction of the functor on $E_\F \Gamma$ coming from the action of $\Gamma$ (by Lemmas~\ref{morphisms} and~\ref{fixed}, $\ul{\gamma}$ maps $(E_\F \Gamma)^F$ to $(E_\F \Gamma)^{F'}$).
The left-hand square commutes at the category level, and the right-hand square commutes by equivariance of $\Phi$.  
This gives us a function 
$$\Psi \co \Map^\Gamma (|E_\F \Gamma|, X) \maps \holim_{\Orb_\F (\Gamma)^\op} X^-,
$$
whose continuity follows from continuity of the restriction maps $\Phi \goesto \Phi^F$.

To prove $\Psi$ is a bijection, we must check that for each natural family of maps 
$$f_F :  |\Orb_\F (\Gamma)^\op \down \Gamma/F| \maps X^F,$$
there exists a unique equivariant map $f: |E_\F \Gamma| \to X$ such that $f \circ |j_F| = f_F$ for each $F\in \F$.
A simplex in the nerve of $E_\F \Gamma$ is represented by a chain of morphisms 
$$(\Gamma/F_0, \gamma F_0) \srm{\phi_1} (\Gamma/F_1, \phi_1 (\gamma F_0)) \srm{\phi_2} \cdots \srm{\phi_n}
	(\Gamma/F_n, \phi_n \circ \cdots \circ \phi_1 (\gamma F_0)).
$$
This simplex is the image, under the functor $j_{F_0}$, of the simplex in the nerve of 
$\Gamma/\gamma F_0 \gamma^{-1}\down \Orb_\F (\Gamma)$ corresponding to the diagram
$$\xymatrix{ & \Gamma/ (\gamma F_0 \gamma^{-1}) \ar[dl]_-{[\gamma]} \ar[d]^-{\phi_1 \circ [\gamma]} \ar[drr]^(.6){\,\,\,\,\,\,\,\,\phi_n \circ \cdots \circ \phi_1 \circ [\gamma]}\\
\Gamma/F_0 \ar[r]^{\phi_1} & \Gamma/F_1  \ar[r]^{\phi_2}  &\cdots \ar[r]^(.4){\phi_n} &\Gamma/F_n,
}
$$
where $[\gamma] = \gamma F_0$ in the notation from Lemma~\ref{morphisms}.
Thus if $f: |E_\F \Gamma| \to X$ satisfies $f \circ |j_F| = f_F$ for each $F\in \F$, then $f$ is completely determined by the maps $f_F$, and we just need to check that the prescribed values on each simplex in fact paste together to give an equivariant map $f: |E_\F \Gamma| \to X$.  Using naturality of the maps $f_F$, one  checks that our prescription for $f$ is well-defined on each simplex and equivariant, and it is immediate from the construction that the simplicial identities are respected.  Hence 
$\Psi$ is a bijection.  

Continuity of $\Psi^{-1}$
follows from continuity of its adjoint map
$$(\Psi^{-1})^\vee \co |E_\F \Gamma| \cross \holim_{\Orb_\F (\Gamma)^\op} X^-  \maps X,$$
which can be checked directly on each closed cell of 
$|E_\F \Gamma|$.
\end{proof}

The following corollary was first proven in Rosenthal \cite[Lemma 2.1]{Rosenthal-L} using a cell-by-cell argument.

\begin{corollary} Let $X$ and $Y$ be $\Gamma$--spaces, and let $\E_\F \Gamma$ be a classifying space for the family $\F$.
If $f: X\to Y$ is a $\Gamma$--equivariant map that induces a weak equivalence $X^F \xmaps{\heq} Y^F$ for each $F\in \F$, then 
the induced map
$$X^{h_\F \Gamma} \maps Y^{h_\F \Gamma}$$
is a weak equivalence as well.
\end{corollary}

\begin{proof} Any two classifying spaces for $\F$ are $\Gamma$--homotopy equivalent, so we may use the categorical model $|E_\F \Gamma|$.  By Theorem \ref{holim}, it is enough to consider the map
$$\holim_{\Orb_\F (\Gamma)^\op} X^- \maps \holim_{\Orb_\F (\Gamma)^\op} Y^-$$
induced by $f$.  But the map of diagrams $X^- \to Y^-$ is a point-wise weak equivalence, so the result follows from the general fact that a point-wise weak equivalence of diagrams induces a weak equivalence between their homotopy limits.
\end{proof}

For applications to algebraic $K$--theory, it is more useful to have versions of these results for $\Omega$--spectra with (naive) $\Gamma$--actions, and in fact this is the context considered in~\cite{Rosenthal}.  Since homotopy limits and equivariant mapping spaces are formed level-wise, the results in this section immediately extend to $\Omega$--spectra.

\vspace{.15in}
We end with another example of how Theorem~\ref{holim}, together with general facts about homotopy limits, may be applied to study generalized homotopy fixed sets. 

\begin{proposition} \label{Syl-prop} 
Let $G$ be a finite group containing a unique (normal) Sylow $p$--subgroup $P$.  Let $W = G/P$ denote the Weyl group of $P$.  Then for any $G$--space $X$, 
the generalized homotopy fixed-set $X^{h_\mcP G}$ $($where $\mcP$ is the family of $p$--subgroups of $G$$)$ is weakly equivalent to the $($ordinary$)$ homotopy fixed set $(X^P)^{h W}$.
\end{proposition}

\begin{proof}
We use the
the cofinality theorem for homotopy limits \cite[XI.9.2]{Bousfield-Kan}, which states that a functor $\mcG\co \C\to \D$ induces  weak equivalences $\holim_\D \ul{X} \maps \holim_\C   \ul{X}\circ \mcG$ (for all diagrams $\ul{X}$) if $\mcG$ is \e{left cofinal} in the sense that the fiber $\mcG\down D$ is contractible for every object $D\in \D$.  Here $\mcG \down D$ is the category in which objects are pairs $(C, \phi)$ with $\phi\co \mcG(C)\to D$, and in which morphisms from  $(C, \phi)$ to  $(C', \phi')$ are maps $\psi\co C\to C'$ such that $\phi'\circ \mcG(\psi) = \phi$.  Note the natural isomorphism
$$(\mcG^\op \down D)^\op \isom D\down \mcG,$$
where $\mcG^\op \co \C^\op \to \D^\op$ is the functor induced by $\mcG$ and $D\down \mcG$ is defined dually to $\mcG\down D$, so that objects in $D\down \mcG$ are pairs $(C, \phi)$ with $\phi\co D\to \mcG(C)$.

Let $\Orb_{P} (G)$ denote the full subcategory of  $\Orb_{\mcP} (G)$ on the single object $G/P$. 
The cofinality theorem  applies to the inclusion $i\co \Orb_{P} (G) \injects   \Orb_{\mcP} (G)$ (or, rather, to the induced functor $i^\op$ between the opposite categories): indeed, for every $p$--subgroup $Q< G$, the fiber $(i^\op \down G/Q)^\op \isom G/Q\down i$  is contractible because every morphism set in $G/Q\down i$ contains exactly one element.  
Theorem~\ref{holim} and cofinality now give a weak equivalence
$$X^{h_\mcP G} \heq \holim_{\Orb_{P} (G)^\op} X^-.$$
Since $\Orb_{P} (G)$ has a single object $G/P$ with automorphism group $W$, the homotopy limit on the right is exactly the homotopy fixed point set for $W$ acting on $X^P$.
\end{proof}

\begin{remark} 
The analogous statement holds $($with the same proof$)$ for any family $\F$ containing a unique maximal element under inclusion.
\end{remark}

\def\cprime{$'$}

\end{document}